\documentclass[11pt]{amsart}
\usepackage{amsthm,amssymb,url,hyperref}
\usepackage{fullpage}

\theoremstyle{plain}
\newtheorem{theorem}{Theorem}[section]

\newtheorem{lemma}[theorem]{Lemma}

\newtheorem{proposition}[theorem]{Proposition}

\theoremstyle{definition}
\newtheorem{definition}[theorem]{Definition}
\newtheorem{problem}[theorem]{Problem}
\newtheorem{example}[theorem]{Example}

\def\aut#1{{\mathrm{Aut}(#1)}}
\def\aff#1{{\mathrm{Aff}(#1)}}
\def\Aff{\mathrm{Aff}}
\def\Z{\mathbb Z}

\title{Non-affine latin quandles of order $2^k$}

\author{Tom\'a\v s Nagy}

\address{Department of Algebra, Faculty of Mathematics and Physics, Charles University, Prague, Czech Republic}
\email{tomas.nagy@email.com}

\begin{document}

\thanks{Research partly supported by the  GA\v CR grant 18-20123S and by the project SVV-2017-260456}



\begin{abstract}
We prove that a non-affine latin quandle (also known as left distributive quasigroup) of order $2^k$ exists if and only if $k=6$ or $k\geq8$. The construction is expressed in terms of central extensions of affine quandles.
\end{abstract}

\maketitle

\section{Introduction}\label{sec:intro}

Latin quandles, also known as left distributive quasigroups, appear in many contexts, including invariants of knots, geometry of symmetric spaces, set-theoretic solutions to the quantum Yang-Baxter equation, or the abstract theory of quasigroups and loops. We refer to \cite{EN,Sta-latin} for an overview of motivations and results on quandles in general, and latin quandles in particular. Many examples and facts mentioned in the introduction are explained in \cite{Sta-latin}.

One source of quandles comes from abelian groups: given an abelian group $A=(A,+)$ and its automorphism $\psi$, set $x*y = (1-\psi)(x)+\psi(y)$. The resulting structure $(A,*)$ is a quandle, called \emph{affine}. It is latin if and only if $1-\psi$ is bijective. Affine quandles were enumerated by Hou \cite{Hou}, providing explicit formulas for orders $p^k$ with $k\leq4$. All latin quandles of order $p$ and $p^2$ are affine \cite{EGS,G}.

A less obvious source of quandles comes from left Bol loops. Given a left Bol loop $L=(L,\cdot)$, set $x*y=x(y^{-1}x)$. The resulting structure is a quandle, called the \emph{core} of $L$. It is latin if and only if $L$ is uniquely 2-divisible (in particular, if it has odd order). The smallest example has order 15, and there is one of order $pq$ ($p>q$ odd primes) if and only if $q\mid p^2-1$ \cite{KNV}. The core of a uniquely 2-divisible left Bol loop is isotopic to the loop itself and therefore cannot be isotopic to any abelian group (especially, it cannot be affine) as loop isotopes of abelian groups are always abelian groups \cite[Corollary III.2.3]{P}. Hence, there are non-affine latin quandles of order $pq$ with $q\mid p^2-1$.

One of the famous Belousov problems \cite{Bel} asked if there is a latin quandle which is not isotopic to a left Bol loop (in particular, it must be non-affine). 
The first counterexample was found by Onoi in 1970 \cite{Onoi}; its order is $2^{16}$. Subsequently, Galkin developed a representation theory for latin quandles over transitive groups, which allowed to settle many problems \cite{Gal}. For example, he proved that the smallest latin quandle not isotopic to a left Bol loop has order 15, and that every smaller latin quandle is affine over an abelian group.

In recent years, there has been a considerable effort in enumeration of quandles in general \cite{AR,JPSZ,VY}, and connected and latin quandles in particular \cite{BB,B???,EGS,G,HSV}.

\begin{problem}
Determine all $n$ such that there exists a non-affine latin quandle of order $n$.
\end{problem}

The following table summarizes the current state of the problem ($k,n\in\mathbb N$, $p,q$ odd primes).

\medskip
\begin{center}\begin{tabular}{|l|l|l|l|}
\hline
order & exists iff & their number & reference \\\hline
$2^k$ & $k=6$ or $k\geq8$ & & Theorem \ref{t:main} \\
$p^k$ & $k\geq3$ & $2p^2-p+2$ of order $p^3$ & \cite{BB,EGS,G} \\
$pq$, $p>q$ & $q\mid p^2-1$ & 2 & \cite{B???} \\
$4n+2$ & (none) & & \cite{Ste} \\
$4p$ & $p\equiv 1\pmod 3$ & 2 & \cite{B???} \\
\hline
\end{tabular}\end{center}
\medskip

Until now, it was unknown what is the smallest $k$ such that there exists a non-affine latin quandle of order $2^k$.
A computer search over the library of transitive groups, based on Galkin's ideas, quickly reveals that $k\geq6$ \cite{HSV}, and Onoi proved that $k\leq16$ \cite{Onoi}. Elaborating Onoi's ideas in the setting of central extensions \cite{BS}, we construct non-affine latin quandles of order $2^k$ for every even $k\geq6$. We also outline how to set a computer search that proves that there are none of order $2^7$. The main theorem and the outline of its proof is stated below.

\begin{theorem} \label{t:main} 
A non-affine latin quandle of order $2^k$ exists if and only $k=6$ or $k\geq8$.
\end{theorem}

\begin{proof}
See Example \ref{ex:2^4k} for orders $2^{4k}$, $k\geq2$, Example \ref{ex:64} for orders $2^{6k}$, $k\geq1$, and use the direct product with the 4-element affine latin quandle or with one of the 8-element ones to obtain the remaining sizes. (The product of an affine and non-affine latin quandle is non-affine, as follows from the Toyoda-Bruck theorem \cite[Theorem 3.1]{Sta-latin} which expresses affineness by an identity.)

For non-existence, use the results of \cite{BS} (summarized in Proposition \ref{p:nilpotence}) to show that every non-affine latin quandle can be represented by a central extension $Q\times_{1-\psi,\psi,\theta} A$, and set a computer search over all parameters $Q,A,\psi,\theta$ to show that all central extensions of order $2^7$ are affine (see Section \ref{s:computation}).
\end{proof}

\section{Preliminaries}

A \emph{latin quandle} $(Q,*)$ is a quasigroup in which all left translations are automorphisms.
The former property says that both equations $a*x=b$ and $y*a=b$ have a unique soution for every $a,b\in Q$. The latter property can be expressed as an identity, called \emph{left self-distributivity}: 
\[ x*(y*z)=(x*y)*(x*z).\] 
A quasigroup is called \emph{medial}, if it satisfies the identity 
\[ (x*y)*(u*v)=(x*u)*(y*v).\]

Let $A=(A,+,-,0)$ be an abelian group and $\psi\in\aut A$ such that $1-\psi$ is bijective (it is indeed a homomorphism). We define a new operation on the set $A$ by
\[ x*y = (1-\psi)(x)+\psi(y). \]

Then $(A,*)$ is a medial latin quandle. Such quandles will be called \emph{affine} over the group $A$, and denoted $\aff{A,\psi}$. The Toyoda-Bruck theorem \cite[Theorem 3.1]{Sta-latin} states that a quasigroup is medial if and only if it is affine in a somewhat broader sense (namely, a quasigroup $(Q, \ast)$ is called affine over an abelian group $(Q, +)$ if there exist automorphisms $\phi$, $\psi$ of $(Q, +)$ with $\phi\psi=\psi\phi$ and $u \in Q$ such that $x \ast y = \phi(x)+\psi(y)+u$ for all $x,y\in Q$). In particular, for latin quandles, mediality and affineness are equivalent properties.

Let $Q$ be a quasigroup, $A$ an abelian group, $\phi,\psi\in\aut A$, and consider a mapping $\theta: Q\times Q \to A$, $(a,b)\mapsto\theta_{a,b}$, called a \emph{cocycle}.
We define an operation on the set $Q\times A$ by
\[ (a,s)*(b,t)=(a*b,\phi(s)+\psi(t)+\theta_{a,b}), \]
for every $a,b\in Q$ and $s,t\in A$. The resulting quasigroup
\[ Q\times_{\phi,\psi,\theta} A=(Q\times A,*) \] 
is called a \emph{central extension} of $Q$ over the triple $(\phi,\psi,\theta)$.

The mapping $Q\times_{\phi,\psi,\theta} A \to Q$, $(a,s)\mapsto a$, is a homomorphism, called \emph{canonical projection}.
For any fixed $e\in Q$ such that $e*e=e$, the mapping $\Aff(A,\phi,\psi)\to Q\times_{\phi, \psi, \theta} A$, $a\mapsto(e,a)$, is a homomorphism, called \emph{canonical injection} over $e$.

\begin{lemma}\label{l:ext}
Let $Q$ be a latin quandle, $A$ an abelian group, $\phi,\psi\in\aut A$ and $\theta:Q\times Q\to A$. Then the central extension $E=Q\times_{\phi,\psi,\theta} A$ is a latin quandle if and only if
$\varphi+\psi=1$, $\theta_{a,a}=0$, and 
\[ \psi(\theta_{b,c})+\theta_{a,b*c}  = \psi(\theta_{a,c})+\phi(\theta_{a,b})+\theta_{a*b,a*c} \tag{LD} \]
for every $a,b,c\in Q$. The extension $E$ is medial if and only if, additionally,
\[ \phi(\theta_{a, b}) + \psi(\theta_{c, d}) + \theta_{a \ast b, c \ast d} = \phi(\theta_{a, c}) + \psi(\theta_{b, d}) + \theta_{a \ast c, b \ast d} \tag{M} \]
for every $a,b,c,d\in Q$. 
\end{lemma}

\begin{proof}
The extension $E$ is a quandle if and only if for all $a, b, c \in Q, x,y,z\in A$:
\[ (a,x)\ast((b,y)\ast(c,z)=((a,x)\ast(b,y))\ast((a,x)\ast(c,z)) \tag{ld} \]
\[ (a,x) \ast(a,x)=(a,x) \tag{i} \]
Rewrite this equations using the definition of central extension and consider only the second coordinate.

Put $x=0$ and from (i) get $\theta_{a,a}=0$. Then, from (i) again, $\phi + \psi = 1$. Put $x=y=z=0$ in (ld) and get (LD).

$E$ is medial if and only if, additionaly, for all $a, b, c, d \in Q, w,x,y,z\in A$
\[ ((a,w)\ast(b,x))\ast((c,y)\ast(d,z))=((a,w)\ast(c,y))\ast((b,x)\ast(d,z)) \tag{m} \]

Rewrite (m) using the definition of $E$, put $w=x=y=z=0$ and get (M) in the second coordinate.

The other implications can be obtained straightforwardly by rewriting (m), (ld) and (i) using the definition of central extension and the equalities (LD), (M), $\theta_{a,a}=0$ and $\phi +\psi = 1$. 
\end{proof}

Let $Q$ be a latin quandle, $A$ an abelian group and $\psi\in\aut A$ such that $\phi=1-\psi$ is bijective. Cocycles satisfying (LD) form a subgroup of the direct power $A^{Q^2}$, to be denoted $\mathrm Z_{LD}(Q,A,\psi)$.

\begin{lemma}\label{lemma01:3}
Let $Q \times_{\phi, \psi, \theta} A$ be a central extension and consider $\alpha \in \aut A$. Then $Q \times_{\phi, \psi, \theta} A$ is isomorphic to $Q \times_{\alpha \phi \alpha^{-1}, \alpha \psi \alpha^{-1}, \alpha \theta} A$.
\end{lemma}

\begin{proof}
It is straightforward to check that $(a, x)\mapsto (a, \alpha(x))$ is an isomorphism.
\end{proof}

Central extensions will be used for our constructions, and also for our non-existence arguments, due to the following fact.

\begin{proposition}\label{p:nilpotence}
Let $Q$ be a~latin quandle of prime power size. Then $Q \simeq F \times_{1-\psi, \psi, \theta} A$ for some latin quandle~$F$ with $|F| < |Q|$, an abelian group~$A$, $\psi\in\aut A$ and $\theta\in\mathrm Z_{LD}(F,A,\psi)$.
\end{proposition}

\begin{proof}
According to~\cite[Corollary 6.6]{BS}, $Q$ is nilpotent, hence there exists a~chain of congruences $0_Q=\alpha_0 < \alpha_1 < \ldots < \alpha_n = 1_Q$ such that $\alpha_{i + 1} / \alpha_i$ is central in $Q/\alpha_i$. Put $F=Q/\alpha_1$. Since $\alpha_1$ is central in $Q$, we have $Q \simeq F \times_{1-\psi, \psi, \theta} A$ for some $A,\psi,\theta$ by~\cite[Proposition 7.8]{BS}.
\end{proof}

\section{Constructions}

\begin{definition}
An algebraic structure $O=(O,+,\cdot,\alpha)$ is called \emph{Onoi ring} if $(O,+,\cdot)$ is a ring (not necessarily associative) such that $a+a=0$ for every $a$, and $\alpha$ is an automorphism of this ring such that 
\[ \alpha^2(a) + \alpha(a) + a = 0 \quad \text{ and } \quad \alpha(a) \cdot b = a \cdot \alpha(b) \]
for every $a, b \in O$.

From $a+a=0$ and $\alpha^2+\alpha+1=0$ it follows that $1-\alpha=\alpha^2$. Hence, the operation 
\[ a*b=\alpha^2(a)+\alpha(b)=(1-\alpha)(a)+\alpha(b)\] 
yields an affine latin quandle, to be denoted $\aff{O}$.
\end{definition}

It is easy to see that $\alpha^3 = 1$ and that the only fixed point of $\alpha$ is 0. Hence, $\alpha$ can be written as a composition of independent 3-cycles and therefore $3\mid(|O|-1)$.
Since $(O,+)$ is an elementary abelian 2-group, finite Onoi rings have $2^{2k}$, $k\geq0$, elements.

\begin{example}
Let $(O,+)$ be an elementary abelian 2-group, $\alpha$ its automorphism satisfying $\alpha^2+\alpha+1=0$ and set $a\cdot b=0$ for every $a,b$. Then $(O,+,\cdot,\alpha)$ is an Onoi ring, called the \emph{zero Onoi ring}.
\end{example}

\begin{example}\label{ex:4}
There are four Onoi rings on four elements, up to isomorphism. Without loss of generality, let $O=\{0,1,2,3\}$ and let $\alpha$ be the 3-cycle $(1\ 2\ 3)$. There are four ways to define a compatible multiplication:
\begin{itemize}
	\item the zero ring, 
	\item the multiplication given by the three element latin quandle on $\{1,2,3\}$ (this is the example that was used by Onoi in his original construction \cite{Onoi}),
	\item the other two examples result from the previous one by a cyclic shift of the rows in the multiplication table.
\end{itemize}
The tables of the operations are below: 

\smallskip
\begin{center}
\begin{tabular}{r||c|c|c|c|}
	+ & 0 & 1 & 2 & 3 \\
	\hline\hline
	0 & 0 & 1 & 2 & 3 \\ 
	\hline
	1 & 1 & 0 & 3 & 2 \\ 
	\hline
	2 & 2 & 3 & 0 & 1 \\ 
	\hline
	3 & 3 & 2 & 1 & 0 \\ 
	\hline
	\end{tabular}
	\qquad
	\begin{tabular}{r||c|c|c|c|}
	$\cdot_1$ & 0 & 1 & 2 & 3 \\
	\hline\hline
	0 & 0 & 0 & 0 & 0 \\ 
	\hline
	1 & 0 & 1 & 3 & 2 \\ 
	\hline
	2 & 0 & 3 & 2 & 1 \\ 
	\hline
	3 & 0 & 2 & 1 & 3 \\ 
	\hline
	\end{tabular}
\end{center}
\smallskip
\begin{center}
	\begin{tabular}{r||c|c|c|c|}
	$\cdot_2$ & 0 & 1 & 2 & 3 \\
	\hline\hline
	0 & 0 & 0 & 0 & 0 \\ 
	\hline
	1 & 0 & 3 & 2 & 1 \\ 
	\hline
	2 & 0 & 2 & 1 & 3 \\ 
	\hline
	3 & 0 & 1 & 3 & 2 \\ 
	\hline
	\end{tabular}
	\qquad
	\begin{tabular}{r||c|c|c|c|}
	$\cdot_3$ & 0 & 1 & 2 & 3 \\
	\hline\hline
	0 & 0 & 0 & 0 & 0 \\ 
	\hline
	1 & 0 & 2 & 1 & 3 \\ 
	\hline
	2 & 0 & 1 & 3 & 2 \\ 
	\hline
	3 & 0 & 3 & 2 & 1 \\ 
	\hline
	\end{tabular}
\end{center}
\smallskip
\end{example}

\begin{example}
Let $O$ be an Onoi ring and $\sigma\in S_n$ a permutation. We define $O^\sigma = (O^n, +_n, -_n, 0_n, \cdot_{\sigma}, \alpha_n)$ where $+_n,-_n,0_n,\alpha_n$ are defined coordinate-wise and  
\[ (a_1, a_2, \ldots, a_n) \cdot_{\sigma} (b_1, b_2, \ldots, b_n) = (a_{\sigma(1)} \cdot b_1, a_{\sigma(2)} \cdot b_2, \ldots, a_{\sigma(n)} \cdot b_n).\]
It is straightforward to check that $O^\sigma$ is an Onoi ring.
\end{example}

\begin{example}
Let $O$ be an Onoi ring and $\sigma\in S_{n\times n}$ a permutation. We define $M_n^{\sigma}(O)$ to be the ring of $n\times n$ matrices over $O$ with standard addition, 
$\alpha$ applied element-wise, and  
\[ (a_{i,j})\cdot_\sigma (b_{i,j})=(\sum_{k=1}^n a_{\sigma(i,k)} \cdot b_{k,j})_{i,j}.\]
It is straightforward to check that $M_n^\sigma(O)$ is an Onoi ring.
(Onoi \cite{Onoi} used this construction for $n=2$ over his 4-element Onoi ring.)
\end{example}

\begin{definition}
Let $O_1,O_2$ be two Onoi rings. A mapping $\mu:O_1^3\to O_2$ is called \emph{an Onoi mapping} between $O_1$ and $O_2$ if it is trilinear (with respect to the additive 2-group) and the following three identities hold:
\begin{align*}
\mu\circ(\alpha_1\times\alpha_1\times\alpha_1)&=\alpha_2\circ\mu,\tag{OM1}\\
\mu\circ(\alpha_1\times1\times1)&=\mu\circ(1\times\alpha_1\times\alpha_1),\tag{OM2}\\ 
\mu\circ(1\times\alpha_1\times1)&=\mu\circ(1\times1\times\alpha_1). \tag{OM3}
\end{align*}
\end{definition}

\begin{example}
If $O$ is an Onoi ring, then $\mu(a,b,c)=a(bc)$ is an Onoi mapping $O^3\to O$. This will be called the \emph{canonical Onoi mapping} for $O$.
\end{example}

\begin{lemma}\label{l:ext_quandle}
Let $O_1,O_2$ be Onoi rings and $\mu:O_1^3\to O_2$ an Onoi mapping. For $a,b\in O_1$, define $\theta_{a,b} = \mu(a, a + b, a + b)$. Then $\theta\in\mathrm Z_{LD}(\aff{O_1},O_2,\psi)$. 
\end{lemma}

\begin{proof}
We will verify the quandle cocycle conditions from Lemma~\ref{l:ext}. To simplify notation, we shall omit the index of $\alpha$ (which is always clear from the context), and we shall write $\alpha a$ instead of $\alpha(a)$ in this proof.

Clearly, $\theta_{a,a}=0$ for all $a \in O_1$ and $\alpha^2+\alpha=1$.
We verify the condition (LD). First, put all terms on one side, and rewrite the cocycle values in terms of the Onoi mapping:
\begin{gather*} 
\alpha\mu(b,b+c,b+c) + \mu(a,a+\alpha^2b+\alpha c,a+\alpha^2 b+\alpha c) \\ 
+ \alpha\mu(a,a+c,a+c) +\alpha^2\mu(a,a+b,a+b)+\mu(\alpha^2 a+\alpha b,\alpha b+\alpha c,\alpha b+\alpha c). 
\end{gather*}
We shall prove that the sum equals 0 for all $a,b,c\in O_1$.

Using linearity in the first coordinate of the last term, we isolate $\mu(\alpha b,\alpha b+\alpha c,\alpha b+\alpha c)=\alpha\mu(b,b+c,b+c)$ and cancel it with the first term.

Using (OM1) we get $\alpha\mu(a,a+c,a+c)=\mu(\alpha a,\alpha (a+c),\alpha (a+c))$. Using (OM2) and linearity of $\alpha$, this equals to $\mu(a,\alpha^2 a+\alpha^2 c,\alpha^2 a+\alpha^2 c)$.

Similarly, using the identities (OM1) and (OM2), we rewrite all the remaining terms to the form $\mu(a,\_,\_)$:
\begin{gather*} 
\mu(a,a+\alpha^2b+\alpha c,a+\alpha^2b+\alpha c) \\ + \mu(a,\alpha^2a+\alpha^2c,\alpha^2 a+\alpha^2 c) 
+\mu(a,\alpha a+\alpha b,\alpha a+\alpha b)+\mu(a,b+c,b+c). 
\end{gather*}
Now, using trilinearity, we expand all terms of the sum so that there is no addition in the arguments of $\mu$. Also, using (OM3), we can shift the $\alpha$ mapping to the last coordinate. The results is the sum of the following terms:
\begin{gather*}
\mu(a,a,a), \mu(a,a,\alpha^2b), \mu(a,a,\alpha c), \mu(a,b,\alpha^2a), \mu(a,b,\alpha b), \\
\mu(a,b,c), \mu(a,c,\alpha a), \mu(a,c,b), \mu(a,c,\alpha^2c), \\ 
\mu(a,a,\alpha a), \mu(a,a,\alpha c), \mu(a,c,\alpha a), \mu(a,c,\alpha c),\\
\mu(a,a,\alpha^2 a), \mu(a,a,\alpha^2 b), \mu(a,b,\alpha^2 a), \mu(a,b,\alpha^2b), \\ 
\mu(a,b,b), \mu(a,b,c), \mu(a,c,b), \mu(a,c,c).
\end{gather*}
Using linearity in the third coordinate and the rules $x+x=0$ and $\alpha^2 x+\alpha x+x=0$, we see that all terms cancel out, thus the cocycle identity holds.
\end{proof}

Hence, the resulting central extension is a latin quandle, to be denoted
\[\mathcal{Q}(O_1,O_2, \mu)=\Aff(O_1) \times_{\alpha_2^2, \alpha_2, \theta} (O_2,+).\]

\begin{lemma}\label{l:ext_medial}
Let $O_1,O_2$ be Onoi rings and $\mu:O_1^3\to O_2$ an \emph{Onoi mapping}. Then the latin quandle $\mathcal{Q}(O_1,O_2, \mu)$ is affine if and only if the following two identities hold for all $a, b, c \in O_1$:
\begin{align*}
\mu(a, b, b) &= \mu(b, a, a), \tag{$\mu$1} \\
\mu(a, b, c) &= \mu(a, c, b). \tag{$\mu$2}
\end{align*}
\end{lemma}

\begin{proof}
We will show that the quandle cocycle condition (M) from Lemma~\ref{l:ext} is equivalent to ($\mu$1) and ($\mu$2). We shall use the same simplification of notation as in the previous proof. Let
\[\Theta(a,b,c,d)=\phi(\theta_{a, b}) + \psi(\theta_{c, d}) + \theta_{a \ast b, c \ast d}\]
denote the left hand side of the identity (M), which then states that $\Theta(a,b,c,d)=\Theta(a,c,b,d)$ for every $a,b,c,d\in O_1$.
Rewriting the cocycle values in terms of the Onoi mapping, we obtain that 
\begin{gather*} 
\Theta(a,b,c,d)=\alpha^2\mu(a,a+b,a+b) + \alpha\mu(c,c+d,c+d) \\ +\mu(\alpha^2a+\alpha b,\alpha^2a+\alpha b+\alpha^2c+\alpha d,\alpha^2a+\alpha b+\alpha^2c+\alpha d), 
\end{gather*}
and using the identities (OM1) and (OM2) and linearity in the first coordinate, it equals
\begin{gather*} 
\Theta(a,b,c,d)=\mu(a,\alpha a+\alpha b,\alpha a+\alpha b) + \mu(c,\alpha^2c+\alpha^2d,\alpha^2c+\alpha^2d) \\ +\mu(a,\alpha a+b+\alpha c+d,\alpha a+b+\alpha c+d) \\ +\mu(b,a+\alpha^2b+c+\alpha^2d,a+\alpha^2b+c+\alpha^2d).
\end{gather*}
Finally, we use linearity in the second and third coordinates to separate variables $a,c$ and $b,d$ in the last two terms, obtaining
\begin{gather*} 
\Theta(a,b,c,d)=\mu(a,\alpha a+\alpha b,\alpha a+\alpha b) + \mu(c,\alpha^2c+\alpha^2d,\alpha^2c+\alpha^2d) \\
+ \mu(a,\alpha a+\alpha c,\alpha a+\alpha c) + \mu(a,\alpha a+\alpha c,b+d) + \mu(a,b+d,\alpha a+\alpha c) \\ + \mu(a,b+d,b+d) + \mu(b,a+c,a+c) + \mu(b,a+c,\alpha^2b+\alpha^2d) \\ + \mu(b,\alpha^2b+\alpha^2d,a+c) +\mu(b,\alpha^2b+\alpha^2d,\alpha^2b+\alpha^2d).
\end{gather*}
Expanding both sides of the identity $\Theta(a,b,c,d)=\Theta(a,c,b,d)$, we see that many terms cancel out, and it becomes equivalent to $\Theta'(a,b,c,d)=\Theta'(a,c,b,d)$, where
\begin{gather*} 
\Theta'(a,b,c,d)= \mu(a,\alpha a+\alpha c,b+d) + \mu(a,b+d,\alpha a+\alpha c) + \mu(a,b+d,b+d) \\
+ \mu(b,a+c,a+c) + \mu(b,a+c,\alpha^2b+\alpha^2d) + \mu(b,\alpha^2b+\alpha^2d,a+c).
\end{gather*}

$(\Leftarrow)$ 
We shall prove that $\Theta'(a,b,c,d)=\Theta'(a,c,b,d)$ for every $a,b,c,d$. Using ($\mu$2), this equality is equivalent to
\[ \mu(a,b+d,b+d)+\mu(b,a+c,a+c)=\mu(a,c+d,c+d)+\mu(c,a+b,a+b).\]
Using linearity and ($\mu$2) again, we see that \[ \mu(x,y+z,y+z)=\mu(x,y,y)+\mu(x,z,z)\] for every $x,y,z$, hence some terms cancel out and the equality is equivalent to
\[ \mu(a,b,b)+\mu(b,a,a)+\mu(b,c,c)=\mu(a,c,c)+\mu(c,a,a)+\mu(c,b,b),\] 
which immediately follows from ($\mu$1).

$(\Rightarrow)$ 
Assume that $\Theta'(a,b,c,d)=\Theta'(a,c,b,d)$ for every $a,b,c,d$.
Writing $\Theta'(x,y,0,y)=\Theta'(x,0,y,y)$, we obtain
\[ \mu(y,x,x)=\mu(x,\alpha x+\alpha y,y)+\mu(x,y,\alpha x+\alpha y)+\mu(x,y,y).\]
Now, using linearity of $\mu$ and the property (OM3), we obtain
\[ \mu(y,x,x)=\mu(x,\alpha x,y)+\mu(x,y,\alpha x)+\mu(x,y,y) \tag{aux1}\]
for every $x,y$.
Similarly, writing $\Theta'(0,x,y,y)=\Theta'(0,y,x,y)$, we obtain
\[ \mu(x,y,y)+\mu(x,y,\alpha^2x)+\mu(x,\alpha^2x,y)=\mu(y,x,x) \tag{aux2}\]
for every $x,y$.
From (aux1) and (aux2), we see that 
\[ \mu(x,y,\alpha^2x)+\mu(x,y,\alpha x)=\mu(x,\alpha^2x,y)+\mu(x,\alpha x,y)\]
for every $x,y$. Using linearity of $\mu$ and the rule $\alpha^2x+\alpha x+x=0$, we see that 
\[ \mu(x,y,x)=\mu(x,x,y) \tag{aux3} \] 
for every $x,y$.
In particular, using (OM3), we have $\mu(x,\alpha x,y)+\mu(x,y,\alpha x)=\mu(x,x,\alpha y)+\mu(x,\alpha y,x)=0$, and (aux1) simplifies to ($\mu$1).

Now, using (aux3) and (OM3), $\mu(a,\alpha a,b+d)=\mu(a,a,\alpha(b+d))=\mu(a,\alpha(b+d),a)=\mu(a,b+d,\alpha a)$. Similarly, $\mu(b,a+c,\alpha^2b)=\mu(a,\alpha^2b,a+c)$ and, using ($\mu$1), $\mu(a,b,b)=\mu(b,a,a)$. Hence, using linearity of $\mu$, 
\begin{gather*} 
\Theta'(a,b,c,d)=\mu(a,\alpha c,b)+\mu(a,\alpha c,d) + \mu(a,b,\alpha c)+\mu(a,d,\alpha c)+ \mu(a,b,d)\\
+\mu(a,d,b)+\mu(a,d,d)+ \mu(b,a,c)+\mu(b,c,a)+\mu(b,c,c)+ \mu(b,a,\alpha^2d)\\
+\mu(b,c,\alpha^2d) + \mu(b,\alpha^2d,a)+\mu(b,\alpha^2d,c).
\end{gather*}
Using (OM3) and ($\mu$1), some terms in the identity $\Theta'(a,b,c,d)=\Theta'(a,c,b,d)$ cancel out, and it becomes equivalent to $\Theta''(a,b,c,d)=\Theta''(a,c,b,d)$, where
\begin{gather*} 
\Theta''(a,b,c,d)= \mu(a,\alpha c,d)+\mu(a,d,\alpha c)+\mu(a,b,d)+\mu(a,d,b)+\mu(b,a,c) \\+\mu(b,c,a)+\mu(b,a,\alpha^2d)+\mu(b,c,\alpha^2d)+\mu(b,\alpha^2d,a)+\mu(b,\alpha^2d,c).
\end{gather*}
Writing $\Theta''(z,x,y,0)=\Theta''(z,y,x,0)$, we obtain
\[ \mu(x,z,y)+\mu(x,y,z)=\mu(y,z,x)+\mu(y,x,z) \tag{aux4}\]  for every $x,y,z$.
Finally, writing the equality $\Theta''(x,y,0,z)=\Theta''(x,0,y,z)$, we obtain
\begin{gather*} 
\mu(x,y,z)+\mu(x,z,y)+\mu(y,x,\alpha^2z)+\mu(y,\alpha^2z,x)=\mu(x,\alpha y,z)+\mu(x,z,\alpha y)
\end{gather*}
for every $x,y,z$. Using (aux4), (OM2) and (OM3) successively, we see that
\begin{gather*} 
\mu(x,\alpha y,z)+\mu(x,z,\alpha y)=\mu(\alpha y,x,z)+\mu(\alpha y,z,x)\\
=\mu(y,\alpha x,\alpha z)+\mu(y,\alpha z,\alpha x)=\mu(y,x,\alpha^2z)+\mu(y,\alpha^2z,x)
\end{gather*}
for every $x,y,z$. Apply this identity and what remains in the equality $\Theta''(x,y,0,z)=\Theta''(x,0,y,z)$ after cancellation is exactly the identity ($\mu$2).
\end{proof}

\begin{example}\label{ex:2^4k}
Let $O$ be an Onoi ring and $e\in O$ such that $e(ee)\neq0$.
Consider the Onoi ring $O^\sigma$ where $\sigma\in S_k$, $k\geq 2$ such that $\sigma(1)=2$ and $\sigma(2)=1$, and denote $e_i=(0,...,0,e,0,...,0)$ where $e$ appears at $i$-th position. 
Let $\mu$ be the canonical Onoi mapping on $O^\sigma$. Then $\mu(e_1,e_1,e_2)=e_1\cdot(0,ee,0,\dots,0)=(0,e(ee),0,\dots,0)$, but $\mu(e_1,e_2,e_1)=e_1\cdot(ee,0,\dots,0)=(0,\dots,0)$, thus violating the identity $(\mu 2)$ from Lemma \ref{l:ext_medial}. Hence, the corresponding central extension $\mathcal Q(O^\sigma,O^\sigma,\mu)$ is a non-affine latin quandle of order $|O|^{2k}$. In particular, using any of the non-zero Onoi rings from Example \ref{ex:4}, we obtain non-affine latin quandles of all orders $(4^k)^2=2^{4k}$, $k\geq2$.
\end{example}

\begin{example}\label{ex:64}
Let $O$ be an Onoi ring and $e\in O$ such that $e(ee)\neq0$. 
Consider the mapping \[ \mu:O^2\times O^2\times O^2\to O,\quad ((a,b),(c,d),(u,v))\mapsto b(du).\]
It is straightforward to verify that this is an Onoi mapping between the direct power $O^2$ and $O$. We have $\mu((0,e),(0,e),(e,0))=e(ee)$, but $\mu((0,e),(e,0),(0,e))=0$, thus violating the identity $(\mu 2)$ from Lemma \ref{l:ext_medial}. Hence, the corresponding central extension $\mathcal Q(O^2,O,\mu)$ is a non-affine latin quandle of order $|O|^{3}$. In particular, using direct powers of Example \ref{ex:4}, we obtain non-affine latin quandles of all orders $(4^k)^3=2^{6k}$, $k \geq 1$.
\end{example}

\section{Non-existence}\label{s:computation}

According to Proposition \ref{p:nilpotence}, latin quandles of order $2^k$ can be constructed from smaller ones by central extensions. Since there are no latin quandles of order 2, the extension is built over $F$, $A$ such that $|F|=2^l$ and $|A|=2^{k-l}$, for some $l\in\{2,\dots,k-2\}$. 

Which groups $A$ are admissible? Indeed, they must possess an automorphism $\psi$ such that $1-\psi$ is bijective. This disqualifies all groups   
$ \Z_{2^{k_1}} \times \Z_{2^{k_2}} \times \ldots \Z_{2^{k_n}}$
where $k_1 > k_2 \geq \ldots \geq k_n$: every automorphism $\psi$ maps the element $(1, 0, \ldots, 0)$ to another element of order $2^{k_1}$, i.e., an element $(a_1, a_2, \ldots, a_n)$ with $a_1$ odd; but then, $(1 - \psi)(1, 0, \ldots, 0) = (1 - a_1, -a_2, \ldots, -a_n)$ with $1 - a_1$ even, so $1 - \psi$ is not bijective.
The group $\Z_4 \times \Z_4 \times \Z_2$ has no admissible automorphism, too, as can be shown by a quick computer calculation.

Existence of non-affine latin quandles of order $2^k$ can be decided by the following algorithm.
\begin{itemize}
	\item[1.] for every abelian group $A$ with $|A|=2^l$, $l\in\{2,\dots,k-2\}$,
	\item[2.] \quad for every $\psi\in\aut A$ up to conjugacy such that $1-\psi$ is bijective, 
	\item[3.] \qquad for every latin quandle $F$ such that $|F|=2^{k-l}$,
	\item[4.] \qquad\quad find a generating set $\Theta$ of the group $\mathrm Z_{LD}(F,A,\psi)$,
	\item[5.] \qquad\quad if every $\theta\in\Theta$ satisfies the cocycle condition (M), answer NO; else answer YES.
\end{itemize}
In step 2., we can take only one automorphism from each conjugacy class of $\aut A$ thanks to Lemma \ref{lemma01:3}.
In step 4., we solve a set of linear equations over the group $A$ with $|F|^2$ indeterminates.

For $k\leq 7$, the only admissible groups $A$ are $\Z_4^2$ and $\Z_2^l$ with $l\in\{2,\dots,k-2\}$, and the quandles $F$ can be taken from the library of small quandles \cite{RIG}. The algorithm is straightforward for the groups $\Z_2^l$

To solve systems of linear equations over the group $\Z_4^2$, one can use the following lemma which is straightforward to prove:

\begin{lemma}
Let $A \in \Z_m^{r\times s}$. Then, $Av=0$ over $\Z_{m}$ if and only if $A^+v^+=0$ over $\Z$ where $A^+=(A\ mI_r)$ and $v^+=\left(\begin{smallmatrix} v \\ x\end{smallmatrix}\right)$ for some $x\in\Z^r$.
\end{lemma}

The basis of the solution space of a set of integral linear equations can be computed by computer calculation using standard techniques. Hence, we can get also the basis of $\mathrm Z_{LD}(F,\Z_4^2,\psi)$. It is straightforward to determine if each basis element satisfies the cocycle condition (M).

It is easy to implement the algorithm in the computer system GAP. It will quickly reveal that there are no non-affine latin quandles of order $2^k$, $k\leq 5$ or $k=7$. It also shows that there are none of order $2^6$ with $A=\Z_4^2$ or $\Z_2^4$, but there are some for the groups $\Z_2^3$ and $\Z_2^2$. The table below shows all triples $(A,F,\psi)$ ($\psi$ up to conjugacy in $\aut A$) such that there is a non-affine latin quandle built by a central extension of $F$ over $A,\psi$. Quandle numbering refers to \cite{RIG}.

\medskip
\begin{center}
\begin{tabular}{|l||l|l|l|l|} \hline
	$F$ & 8\_2 & 8\_3  & 4\_1 $\times$ 4\_1 & 16\_$n$, $n=1,2,4,5,\dots,9$ \\[1mm]\hline
	$A$ & $\Z_2^3$ & $\Z_2^3$  & $\Z_2^2$ & $\Z_2^2$ \\[1mm]\hline
$\psi$ & $\begin{pmatrix}
    1 & 0 & 1\\
    1 & 1 & 1\\
    0 & 1 & 1
\end{pmatrix}$ & $\begin{pmatrix}
    1 & 0 & 1\\
    1 & 1 & 0\\
    0 & 1 & 0
\end{pmatrix}$ & $\begin{pmatrix}
    1 & 1\\
    1 & 0
\end{pmatrix}$  & $\begin{pmatrix}
    1 & 1\\
    1 & 0
\end{pmatrix}$ \\ \hline
\end{tabular}
\end{center}
\medskip

The computations took approximately 2 days of computation time.

The algorithm could be extended to enumeration of non-affine latin quandles of order $2^k$: in step~6., we would filter surviving cocycles up to isomorphism of the corresponding extensions. We tried to adapt some of the isomorphism checking techniques that were used in other projects (such as \cite{DV}), but the dimension of the cocycle space seems to be too large to succeed in the enumeration.

\end{document}